\documentclass[12pt]{amsart}
\usepackage{amssymb,color}
\numberwithin{equation}{section} 

\usepackage{graphicx}

\newcommand{\bea}{\begin{eqnarray}}
\newcommand{\eea}{\end{eqnarray}}
\newcommand{\ba}{\begin{array}}
\newcommand{\ea}{\end{array}}
\newcommand{\edc}{\end{document}}
\newcommand{\bc}{\begin{center}}
\newcommand{\ec}{\end{center}}
\newcommand{\be}{\begin{equation}}
\newcommand{\ee}{\end{equation}}

\def\cf{{\mathcal F}}

\def\bc{{\mathbb C}}

\def\bn{{\mathbb N}}

\def\bq{{\mathbb Q}}
\def\br{{\mathbb R}}

\def\bz{{\mathbb Z}}

\def\a{\alpha}
\def\b{\beta}
\def\g{\gamma}  
\def\d{\delta}

\newtheorem{thm}{Theorem}[section]
\newtheorem{lem}[thm]{Lemma}

\theoremstyle{remark}
\newtheorem{rem}{Remark}[section]
\newtheorem{ex}{Example}[section]

\begin{document}

\title{On unconventional limit sets of contractive functions on $\mathbb Z_p$}

\author{Farrukh Mukhamedov}
\address{Mukhamedov Farrukh\\
 Department of Computational \& Theoretical Sciences\\
Faculty of Science, International Islamic University Malaysia\\
P.O. Box, 141, 25710, Kuantan\\
Pahang, Malaysia} \email{{\tt far75m@yandex.ru} {\tt farrukh\_m@iium.edu.my}}

\author{Otabek Khakimov}
\address{Otabek Khakimov\\
Institute of mathematics, National University of Uzbekistan, 29,
Do'rmon Yo'li str., 100125, Tashkent, Uzbekistan.} \email {{\tt
hakimovo@mail.ru}}

\begin{abstract}
In the present paper, we are going to study metric properties of
unconventional limit set of a semigroup $G$ generated by
contractive functions $\{f_{i}\}_{i=1}^N$ on the unit ball
$\mathbb Z_p$ of $p$-adic numbers. Namely, we prove that the
unconventional limit set is compact, perfect and uniformly
disconnected.  Moreover, we provide an example of two contractions
for which the corresponding unconventional limiting set is
quasi-symmetrically equivalent to the symbolic Cantor set.

\end{abstract}

\maketitle

\section{introduction}

Non-Archimedean dynamical systems is one of the most popular areas
of the modern mathematics. There are many works devoted to
$p$-adic dynamics \cite{An,B1,B2,X1,X2,GKL,KM2,KhN,FM2,RS,Wo}. On
the other hand, in investigations of random mappings it has been
studied the metric properties of limit sets in the Euclidean
spaces (see for example \cite{Bl,BP,Hut,JV}). These investigations
have found their applications in the fractal geometry
\cite{Fal,GS}. This naturally motivates to consider the metric
properties of limit sets in a non-Archimedean setting. In this
direction, very recently, in \cite{X} it has been considered a
semigroup $G$ generated by a finite set $\{f_i\}_{i=1}^N$ of
contractive functions on $\mathcal O$ (here  $\mathcal
O=\{x\in\mathbb K: |x|\leq1\}$ is the closed unit ball of the
non-Archimedean field $\mathbb K$). Namely,
$G=\bigcup\limits_{k\geq1}G_k$, where $G_k=\{f_{i_1}\circ
f_{i_2}\circ\dots\circ f_{i_k}, 1\leq i_j\leq N, 1\leq j\leq k\}$.
Furthermore, it was studied matric properties of the limit set
$\Lambda$ of $G$ which is a complement of the set of all points
$x\in\mathcal O$ for which there exist open neighborhoods $U_x$ of
$x$ such that $g(U_x)\cap U_x=\varnothing$ for all but finitely
many $g\in G$. Note that the limit set $\Lambda$ of $G$ is a very
important object in the study of random dynamical systems (see for
example \cite{X2,Sil1,Wo}).

It is known that that the composition of two contractive mappings
is also contraction. But, in general, the product or the sum of
such kind of contractions is contraction. But in a non-Archimedean
case they are also contractions. Using that fact in \cite{MA2015}
it has been studied the uniqueness limiting set of unconventional
iterates of contractive mappings. Note that these results play an
important role in the theory of $p$-adic Gibbs measures
\cite{GRR},\cite{FM1}-\cite{M2013}.

In the present paper, we are going to study metric properties of
unconventional limit set of contractive function on $\mathbb Z_p$.
Namely, we prove that the unconventional limit set is compact,
perfect and uniformly disconnected. Moreover, we provide an
example of two contractions for which the corresponding
unconventional limiting set is quasi-symmetrically equivalent to
the symbolic Cantor set. This paper can be considered as a
generalization and extension of some results of \cite{X}.

\section{definitions and preliminary results}

A metric space $X$ is said to be {\it doubling} if there is a
constant $k$ such that every disk $B$ in $X$ can be covered with
at most $k$ disks of half the radius of $B$. A number of metric
spaces have this property, e.g. the Euclidean space, a compact
Riemann surface, etc. However a non-Archimedean space is not
necessary a doubling space, e.g. $\mathbb C_p$. Hence it is very
important to know whether a subspace of a non-Archimedean space is
a doubling space. Note that $\mathbb Q_p$ is a doubling. Let
$(X,d)$ be a complete metric space. Let $\mathcal B$ be the
collection of all bounded subsets of $X$. For a set $E\in\mathcal
B$, we denote the diameter of $E$ by $\mbox{diam}(E)=\sup_{z,w\in
E}d(z,w)$. By definition, a set $E\in\mathcal B$ is called a {\it
uniformly perfect set} if $E$ contains at least two points and
there exists a constant $c>0$ such that for any point $x_0\in E$
and $0 <r<\mbox{diam}(E)$, the annulus $\{x\in X: cr\leq
d(x,x_0)\leq r\}$ meets $E$. We say that a metric space $(X,d)$ is
{\it uniformly disconnected} if there is a constant $C>1$ so that
for each $x\in X$ and $r>0$ we can find a closed subset $A$ of $X$
such that $B_{r/C}(x)\subset A\subset B_r(x)$, and
$\mbox{dist}(A,X\setminus A)\geq C^{-1}r$. Let $(X,d_X)$, $(Y,d_Y
)$ be metric spaces. A mapping $f :X\to Y$ is said to be {\it
quasi-symmetric} if it is not constant and if there exists a
homeomorphism $\eta:[0,+\infty)\to[0,+\infty)$ such that for any
points $x,y,z\in X$ and $t>0$ we have that $d_X(x,y)\leq
td_X(x,z)$ implies that $d_Y(f(x),f(y))\leq\eta(t)d_Y(f(x),f(z))$.
Two metric spaces are said to be {\it quasi-symmetrically
equivalent} if there is a quasi-symmetric mapping from one to the
other.

Let $F =\{0,1\}$ and $F^\infty$ denote the set of sequences $\{x_i\}_{i=1}^\infty$ with $x_i\in F$ for each $i$.
Given $x =\{x_i\}$ and
$y=\{y_i\}$ in $F^\infty$, let $L(x,y)$ be the largest integer $l$ such that $x_i=y_i$ when $1\leq i\leq l$,
and set $L(x,y)=\infty$
when $x=y$. Take $0<a<1$ and set $d_a(x,y)=a^{L(x,y)}$. Then $F^\infty$ endowed with the metric $d_a(x,y)$ is a
standard symbolic Cantor set. The space $Z_2$ of $2$-adic integers can be regarded as the standard symbolic
Cantor set $F^\infty$ endowed with the metric $d_{1/2}(x, y)$.

\begin{thm}\label{qseC}\cite{GS} Suppose that $(X,d)$ is a compact metric space which is bounded, complete, doubling,
uniformly disconnected, and uniformly perfect. Then $X$ is
quasi-symmetrically equivalent to the symbolic Cantor set
$F^\infty$, where we take $F =\{0,1\}$ with the metric
$d_a(x,y)=a^{L(x,y)}$ where $a=1/2$.
\end{thm}

\subsection{$p$-adic numbers}

In what follows $p$ will be a fixed prime number. The set $\bq_p$ is
defined as a completion of the rational numbers $\bq$ with respect
to the norm $|\cdot|_p:\bq\to\br_+$ given by
\begin{eqnarray}
|x|_p=\left\{
\begin{array}{c}
  p^{-r} \ x\neq 0,\\
  0,\ \quad x=0,
\end{array}
\right.
\end{eqnarray}
here, $x=p^r\frac{m}{n}$ with $r,m\in\bz,$ $n\in\bn$,
$(m,p)=(n,p)=1$. The absolute value $|\cdot|_p$ is non-Archimedean,
meaning that it satisfies the strong triangle inequality $|x + y|_p
\leq \max\{|x|_p, |y|_p\}$. We recall a nice property of the norm,
i.e. if $|x|_p>|y|_p$ then $|x+y|_p=|x|_p$. Note that this is a
crucial property which is proper to the non-Archimedenity of the
norm.

Any $p$-adic number $x\in\bq_p$, $x\neq 0$ can be uniquely represented in the form
\begin{equation}\label{canonic}
x=p^{\g(x)}(x_0+x_1p+x_2p^2+...),
\end{equation}
where $\g=\g(x)\in\bz$ and $x_j$ are integers, $0\leq x_j\leq
p-1$, $x_0>0$, $j=0,1,2,\dots$ In this case $|x|_p=p^{-\g(x)}$.
Note that the basics of $p$-adic analysis are given in \cite{R}.

The next lemma immediately follows from the properties of the
$p$-adic norm

\begin{lem} Let $f$ be a mapping on $\bq_p$. Then the following statements are equivalent\\
$(a)$ $f$ is contraction;\\
$(b)$ $|f(x)-f(y)|_p<|x-y|_p$ for all $x,y\in\bq_p$;\\
$(c)$  $|f(x)-f(y)|_p\leq\frac{1}{p}|x-y|_p$ for all $x,y\in\bq_p$.
\end{lem}

For each $a\in \bq_p$, $r>0$ we denote
$$
B_r(a)=\{x\in \bq_p : |x-a|_p\leq r\}, \ \  B_r^-(a)=\{x\in \bq_p
: |x-a|_p< r\} $$ and the set of all {\it $p$-adic integers}
$\bz_{p}=\left\{ x\in \bq_{p}:\ |x|_{p}\leq1\right\}$.

\section{unconventional limit set}

Let $M,N,L$ be fixed positive integers. For convenience, instead
of $\{1,2,...,N\}$ we will write $[1,N]$. Let consider a  family
$\xi:=\{\xi_{ij}:[1,N]\to[1,N]: \ (i,j)\in[1,M]\times[1,L]\}$ of
mappings. Let $\{f_i\}_{i=1}^N$ be a finite set of contractive
mappings on $\mathbb Z_p$. In what follows, we frequently use the
denotation $\Sigma=[1,N]^\bn$.

For each $\a\in\Sigma$, $n\in\bn$ we denote
\begin{equation}\label{Falphan1}
F_{\alpha,n}=\sum\limits_{i=1}^M\prod_{j=1}^LF_{\alpha,n}^{\xi_{ij}}.
\end{equation}
where
\begin{equation}\label{Falphan}
F_{\alpha,n}^{\xi_{ij}}=f_{\xi_{ij}(\alpha_1)}\circ\cdots\circ
f_{\xi_{ij}(\alpha_n)}, \ \ \a=(\a_1,\dots,\a_n,\dots).
\end{equation}

Put
$$
\cf_\xi=\bigcup\limits_{n\geq1}\cf_{\xi,n}, \ \
\cf_{\xi,n}=\left\{F_{\alpha,n}:\ \alpha\in\Sigma\right\}.
$$

The set $\cf_\xi$ is called a \textit{unconventional set} of the
semigroup $G$ generated by a finite set $\{f_i\}_{i=1}^N$.

\begin{rem} In the sequel, we always assume that the family $\xi$ satisfies the following condition
\begin{equation}\label{xi(a)}
\bigcup\limits_{k=1}^N\bigcup\limits_{i=1}^M\bigcup\limits_{j=1}^L\{\xi_{ij}(k)\}=[1,N].
\end{equation}
Otherwise, the set $\cf_\xi$ will be generated by a subset of
$\{f_i\}_{i=1}^N$.
\end{rem}

\begin{ex} Let us construct an example of a family of mappings $\xi=\{\xi_{ij}:[1,N]\to[1,N]\}$ which satisfies \eqref{xi(a)}.
For any integer number $\ell\in[1,N]$ we define an action on
$[1,N]$ by
$$
(\ell\ast k)=\left\{\begin{array}{ll}
(\ell+k)(\operatorname{mod }N)&\mbox{if }\ N\not|(\ell+k)\\
N&\mbox{if }\ N|(\ell+k).
\end{array}\right. \ \ k\in[1,N].
$$
Then for any number $\xi_{ij}\in[1,N]$ we define
$\xi_{ij}(k):=(\xi_{ij}\ast k)$, $k\in[1,N]$. It is clear that
$\xi_{ij}:[1,N]\to[1,N]$ and  \eqref{xi(a)} holds.
\end{ex}

\begin{lem}\label{Flemma} For any $F\in\cf_\xi$ the
following statements hold:\\
$(a)$ $F:\bz_p\to\bz_p$;\\
$(b)$ $F$ is a contraction.
\end{lem}
\begin{proof}
Let $F\in\cf_\xi$. By construction of the set $\cf_\xi$ there
exist $\alpha\in\Sigma$ and an integer number $n\geq1$ such that
\begin{equation}\label{F=}
F=\sum_{i=1}^{M}\prod_{j=1}^{L}F_{\alpha,n}^{\xi_{ij}}.
\end{equation}

$(a)$ Since $f_k:\bz_p\to\bz_p,\ k=\overline{1,N}$ by
\eqref{Falphan} we have $F_{\alpha,n}^{\xi_{ij}}:\bz_p\to\bz_p$
for any $(i,j)\in[1,M]\times[1,L]$. Consequently, for any
$x\in\bz_p$ using the strong triangle inequality from \eqref{F=}
one can find
$$
|F(x)|_p\leq\max\limits_{1\leq i\leq M}\left|\prod\limits_{j=1}^{L}F_{\alpha, n}^{\xi_{ij}}(x)\right|_p
\leq\max\limits_{1\leq i\leq M\atop{1\leq j\leq L}}\left|F_{\alpha, n}^{\xi_{ij}}(x)\right|_p\leq1.
$$

$(b)$ Let $x,y\in\bz_p$. Then we have
\begin{eqnarray}\label{F(x)-F(y)}
F(x)-F(y)&=&\sum\limits_{i=1}^{M}\prod\limits_{j=1}^{L}F_{\alpha, n}^{\xi_{ij}}(x)-
\sum\limits_{i=1}^{M}\prod\limits_{j=1}^{L}F_{\alpha, n}^{\xi_{ij}}(y)\nonumber\\
&=&\sum\limits_{i=1}^{M}\sum\limits_{j=1}^{L}\left(F_{\alpha, n}^{\xi_{ij}}(x)
-F_{\alpha, n}^{\xi_{ij}}(y)\right)\prod\limits_{k>j}F_{\alpha, n}^{\xi_{ik}}(x)
\prod\limits_{l<j}F_{\alpha, n}^{\xi_{il}}(y)
\end{eqnarray}
Again by means of the strong triangle inequality from
\eqref{F(x)-F(y)}, one gets
\begin{equation}\label{normF(x)-F(y)}
\left|F(x)-F(y)\right|_p\leq\max\limits_{1\leq i\leq M\atop{1\leq j\leq L}}
\left|F_{\alpha, n}^{\xi_{ij}}(x)-F_{\alpha, n}^{\xi_{ij}}(y)\right|_p
\end{equation}
By \eqref{Falphan} we have
$$
\left|F_{\alpha, n}^{\xi_{ij}}(x)-F_{\alpha, n}^{\xi_{ij}}(y)\right|_p\leq\frac{1}{p}|x-y|_p\ \ \ \mbox{for any }
(i,j)\in[1,M]\times[1,L].
$$
Now substituting the last one into \eqref{normF(x)-F(y)} one finds
the required assertion.
\end{proof}

A main aim of this paper is to study limiting set of $\cf_\xi$. By
{\it unconventional limit set} of $\cf_\xi$ is meant the set
$\Lambda^\xi$ which the compliment of the {\it discontinuity set}
$\Omega^\xi$, i.e. $\bz_p\setminus\Omega^\xi$. The discontinuity
set $\Omega^\xi\subset\bz_p$ of $\cf_\xi$ is defined as follows:
$x\in\Omega^\xi$ if and only if there is a disk $B_r(x)$ such that
there are only finitely many $F\in\cf_\xi$ satisfying
$F(B_r(x))\cap B_r(x)=\emptyset$.

\section{some properties of the unconventional limit set}

In this section we study several metric properties of the set
$\Lambda^\xi$.

Denote
\begin{equation}\label{L_0}
\Lambda_0^\xi=\left\{\sum_{i=1}^{M}\prod_{j=1}^{L}x_{\xi_{ij},\alpha}^{(n)}: x_{\xi_{ij},\alpha}^{(n)}=F_{\alpha,n}^{\xi_{ij}}(x_{\xi_{ij},\alpha}^{(n)})
\mbox{ for some }\alpha\in\Sigma\mbox{ and }n\in\bn\right\}.
\end{equation}

\begin{thm}\label{teor1}
The limit set $\Lambda^\xi$ of $\cf_{\xi}$ coincides with the
closure of $\Lambda_0^\xi$.
\end{thm}
\begin{proof}
Let us first show that $\Lambda^\xi$ is closed. It is enough to
establish that $\Omega^\xi$ is open. Take any $x\in\Omega^\xi$.
Then there exist $r>0$ and $\{F_k\}_{k=1}^m\subset\cf_\xi$ such
that $F_k(B_r(x))\cap B_r(x)=\emptyset,\ k=\overline{1,m}$. Since
for any $y\in B^-_r(x)$, one has $B_r(y)=B_r(x)$, hence we have
$F_k(B_r(y))\cap B_r(y)=\emptyset$. This implies that
$B^-_r(x)\subset\Omega^\xi$, so $\Omega^\xi$ is open.

Let $x\in\Lambda_0^\xi$. Then there exist $\alpha\in\Sigma$ and an
integer number $n\geq1$ such that
\begin{equation}\label{pred_x}
x=\sum_{i=1}^{M}\prod_{j=1}^{L}x_{\xi_{ij},\alpha}^{(n)},
\end{equation}
where $x_{\xi_{ij},\alpha}^{(n)}$ is a fixed point of $F_{\alpha, n}^{\xi_{ij}}$.

Consider a sequence $\{F_m\}_{m=1}^\infty$ defined by
\begin{equation}\label{F_m}
F_m=\sum\limits_{i=1}^{M}\prod\limits_{j=1}^{L}(F_{\alpha, n}^{\xi_{ij}})^m=
\sum\limits_{i=1}^{M}\prod\limits_{j=1}^{L}\left(f_{\xi_{ij}(\alpha_1)}\circ\cdots\circ f_{\xi_{ij}(\alpha_n)}\right)^m.
\end{equation}
It is clear that for any $m\geq1$, we have $F_m\in\cf_{\xi, nm}$,
hence $F_m\in\cf_{\xi}$. Take any $r>0$ and $y\in B_r(x)$. Then
from \eqref{pred_x} and \eqref{F_m} one finds
\begin{eqnarray*}
F_m(y)-x&=&\sum\limits_{i=1}^{M}\prod\limits_{j=1}^{L}(F_{\alpha, n}^{\xi_{ij}})^m(y)-\sum\limits_{i=1}^{M}\prod\limits_{j=1}^{L}x_{\xi_{ij},\alpha}^{(n)}\\
&=&\sum\limits_{i=1}^{M}\prod\limits_{j=1}^{L}(F_{\alpha, n}^{\xi_{ij}})^m(y)-\sum\limits_{i=1}^{M}\prod\limits_{j=1}^{L}F_{\alpha, n}^{\xi_{ij}} (x_{\xi_{ij},\alpha}^{(n)})\\
&=&\sum\limits_{i=1}^{M}\prod\limits_{j=1}^{L}(F_{\alpha, n}^{\xi_{ij}})^m(y)-\sum\limits_{i=1}^{M}\prod\limits_{j=1}^{L}(F_{\alpha, n}^{\xi_{ij}})^m (x_{\xi_{ij},\alpha}^{(n)})\\
&=&\sum_{i=1}^{M}\sum_{j=1}^{L}\left[(F_{\alpha, n}^{\xi_{ij}})^m(y)-(F_{\alpha, n}^{\xi_{ij}})^m(x_{\xi_{ij},\alpha}^{(n)})\right]
\prod_{k>j}(F_{\alpha, n}^{\xi_{ik}})^m(y)\prod_{l<j}(F_{\alpha, n}^{\xi_{il}})^m(x_{\xi_{il},\alpha}^{(n)}).
\end{eqnarray*}
The last equality with the strong triangle inequality implies that
\begin{equation}\label{normF_m(y)-x}
\left|F_m(y)-x\right|_p\leq\max_{i,j}\left|(F_{\alpha, n}^{\xi_{ij}})^m(y)-(F_{\alpha, n}^{\xi_{ij}})^m(x_{\xi_{ij},\alpha}^{(n)})\right|_p.
\end{equation}
The contractivity  $F_{\alpha, n}^{\xi_{ij}}$ with
\eqref{normF_m(y)-x} yields
$$
\left|F_m(y)-x\right|_p\leq\frac{1}{p^m}\max_{i,j}\left|y-x_{\xi_{ij},\alpha}^{(n)}\right|_p\leq\frac{1}{p^m}.
$$
Then there exists a positive integer number $m_r$ such that
$F_m(y)\in B_r(x)$ for all $m>m_r$. This means that
$F^{m}(B_r(x))\cap B_r(x)\neq\emptyset$. Consequently,
$x\in\Lambda^\xi$. Since $\Lambda^\xi$ is closed, we have
$\overline{\Lambda}_0^\xi\subset\Lambda^\xi$.

Now suppose that $x_0\not\in\overline{\Lambda}_0^\xi$. Then there
exists $r>0$ such that
$B_r(x_0)\cap\overline{\Lambda}_0^\xi=\emptyset$.
 Choose a
positive integer $n_0$ such that $\frac{1}{p^{n_0}}<r$. Consider a
function $F\in\cf_\xi$ defined by
$$
F=\sum_{i=1}^{M}\prod_{j=1}^{L}F_{\alpha, n_0}^{\xi_{ij}}\ \ \
\mbox{for some }\alpha\in\Sigma.
$$
It is easy to see that
\begin{equation}\label{Fp^n}
|F(x)-F(y)|_p\leq\frac{1}{p^n}|x-y|_p\leq\frac{1}{p^n}\ \ \ \mbox{for any } x,y\in\bz_p.
\end{equation}

Denote
$$
x_F=\sum_{i=1}^{M}\prod_{j=1}^{L}x_{\xi_{ij},\alpha}^{(n_0)},
$$
here as before  $x_{\xi_{ij},\alpha}^{(n_0)}$ is a fixed point of
$F_{\alpha, n_0}^{\xi_{ij}}$. Then $x_F\in\Lambda_0^\xi$.
According to Lemma \ref{Flemma} the function $F$ has a unique
fixed point $z_F$ on $\bz_p$. Now from the strong triangle
inequality and \eqref{Fp^n} we obtain
\begin{eqnarray}\label{z_F-x_F}
|z_F-x_F|_p&=&\left|F(z_F)-x_F\right|_p\nonumber\\
&=&\left|\sum_{i=1}^{M}\prod_{j=1}^{L}F_{\alpha,
n_0}^{\xi_{ij}}(z_F)-
\sum_{i=1}^{M}\prod_{j=1}^{L}F_{\alpha, n_0}^{\xi_{ij}}(x_{\xi_{ij},\alpha}^{(n_0)})\right|_p\nonumber\\
&=&\left|\sum_{i=1}^{M}\sum_{j=1}^{L}\left[F_{\alpha,
n_0}^{\xi_{ij}}(z_F)- F_{\alpha,
n_0}^{\xi_{ij}}(x_{\xi_{ij},\alpha}^{(n_0)})\right]\prod_{k>j\atop{l<j}}F_{\alpha,
n_0}^{\xi_{ik}}(z_F)
F_{\alpha, n_0}^{\xi_{il}}(x_{\xi_{il},\alpha}^{(n_0)})\right|_p\nonumber\\
&\leq&\max_{i,j}\left|F_{\alpha, n_0}^{\xi_{ij}}(z_F)-F_{\alpha,
n_0}^{\xi_{ij}}(x_{\xi_{ij},\alpha}^{(n_0)})\right|_p\prod_{k>j\atop{l<j}}\left|F_{\alpha,
n_0}^{\xi_{ik}}(z_F)F_{\alpha,
n_0}^{\xi_{il}}(x_{\xi_{il},\alpha}^{(n_0)})\right|_p\leq\frac{1}{p^{n_0}}<r.
\end{eqnarray}
For any $y\in B_r^-(x_0)$ due to contractivit of $F$ one gets
$$
|F(y)-z_F|_p=|F(y)-F(z_F)|_p<|y-z_F|_p.
$$
Again the strong triangle inequality implies
$$
|F(y)-y|_p=|F(y)-z_F+z_F-y|_p=|y-z_F|_p.
$$
Since $y\in B_r^-(x_0)$ and $B_r^-(x_0)\cap\overline{\Lambda}_0^\xi=\emptyset$, $x_F\in\Lambda_0^\xi$
one concludes that $|y-x_F|_p>r$. Therefore, from \eqref{z_F-x_F} it follows that
$$
|y-z_F|_p=|y-x_F+x_F-z_F|_p=|y-x_F|_p>r.
$$
Hence,
$$
|F(y)-y|_p=|y-x_F|_p>r.
$$
Consequently, with $|x_0-y|_p<r$ one finds
$$
|F(y)-x_0|_p=|F(y)-y+y-x_0|_p=|F(y)-y|_p>r.
$$
This means that $F(y)$ does not belongs to the disk $B_r(x_0)$.
Hence, $F(B_r(x_0))\cap B_r(x_0)=\emptyset$ which implies
$x_0\in\Omega^\xi$. It follows that
$\Lambda^\xi\subset\overline{\Lambda}_0^\xi$. Hence
$\Lambda^\xi=\overline{\Lambda}_0^\xi$. This completes the proof.
\end{proof}

Now we need the following auxiliary fact.

\begin{lem}\label{xaaa}
Let $\a\in\Sigma$. Then a sequence defined by
\begin{equation}\label{sequence}
x_{\alpha}^{(n)}=\sum_{i=1}^{M}\prod_{j=1}^{L}x_{\xi_{ij},\alpha}^{(n)}
\end{equation}
converges as $n\to\infty$. Here as before
$x_{\xi_{ij},\alpha}^{(n)}$ be a fixed point of $F_{\alpha,
n}^{\xi_{ij}}$.
\end{lem}
\begin{proof}
Due to the closedness of $\bz_p$ it is enough to show that the sequence
\eqref{sequence} is Cauchy.

Take an arbitrary $\varepsilon>0$ and choose $n_0\in\bn$ such that $\frac{1}{p^{n_0}}<\varepsilon$.
Then for any $n,m\geq n_0\ (n>m)$ we have
\begin{eqnarray*}
\left|x_\alpha^{(n)}-x_\alpha^{(m)}\right|_p&=&\left|\sum_{i=1}^M\prod_{j=1}^Lx_{\xi_{ij},\alpha}^{(n)}-
\sum_{i=1}^M\prod_{j=1}^Lx_{\xi_{ij},\alpha}^{(m)}\right|_p\\
&\leq&\max_{i,j}\left|x_{\xi_{ij},\alpha}^{(n)}-x_{\xi_{ij},\alpha}^{(m)}\right|_p=
\max_{i,j}\left|F_{\alpha,n}^{\xi_{ij}}(x_{\xi_{ij},\alpha}^{(n)})-F_{\alpha,m}^{\xi_{ij}}(x_{\xi_{ij},\alpha}^{(m)})\right|_p\\
&\leq&\frac{1}{p^m}\max_{i,j}\left|\left(f_{\xi_{ij}(\a_{m+1})}\circ\cdots\circ f_{\xi_{ij}(\a_{n})}\right)
(x_{\xi_{ij},\alpha}^{(n)})-x_{\xi_{ij},\alpha}^{(m)}\right|_p\leq\frac{1}{p^m}<\varepsilon.
\end{eqnarray*}
This means that $\{x_{\alpha}^{(n)}\}$ is a Cauchy sequence.
\end{proof}
For a given $\alpha\in\Sigma$ due to Lemma \ref{xaaa} we denote
$$
x_\alpha=\lim_{n\to\infty}x_{\alpha}^{(n)}.
$$
Put
$$
\widetilde{\Lambda}^\xi=\{x_\alpha: \alpha\in\Sigma\}.
$$
\begin{thm}\label{Lambda=tildeLambda}
One has $\Lambda^\xi=\widetilde{\Lambda}^\xi$.
Moreover, $\Lambda^\xi$ is compact.
\end{thm}
\begin{proof}
Let us first show that $\widetilde{\Lambda}^\xi$ is compact. To do
so, we define a mapping $\pi:\Sigma\to\widetilde{\Lambda}^\xi$ as
follows:
$$
\pi(\alpha)=x_\alpha.
$$
It is know that $\Sigma$ is compact and to establish compactness
of $\widetilde{\Lambda}^\xi$ it is sufficient to show that $\pi$
is continuous. Suppose that
$\alpha=(\alpha_1,\dots,\alpha_k,\dots)\in\Sigma$ and
$\varepsilon>0$. Choose $n_0\in\bn$ such that
$\frac{1}{p^{n_0}}<\varepsilon$. Then for any $\beta\in\Sigma$
with $\beta_k=\alpha_k,\ k\leq n_0$, we find
\begin{equation}\label{lim1}
\begin{array}{ll}
\left|\sum\limits_{i=1}^M\prod\limits_{j=1}^LF_{\beta,n_0}^{\xi_{ij}}(x_{\xi_{ij},\beta}^{(n_0)})
-\sum\limits_{i=1}^M\prod\limits_{j=1}^LF_{\alpha,n}^{\xi_{ij}}(x_{\xi_{ij},\alpha}^{(n)})\right|_p\\[3mm]
=\left|\sum\limits_{i}\sum\limits_j\left[F_{\beta,n_0}^{\xi_{ij}}(x_{\xi_{ij},\beta}^{(n_0)})
-F_{\alpha,n}^{\xi_{ij}}(x_{\xi_{ij},\alpha}^{(n)})\right]\prod\limits_{k>j\atop{l<j}}
F_{\beta,n_0}^{\xi_{ik}}(x_{\xi_{ik},\beta}^{(n_0)})F_{\alpha,n}^{\xi_{il}}(x_{\xi_{il},\alpha}^{(n)})\right|_p\\[3mm]
\leq\max\limits_{i,j}\left|F_{\beta,n_0}^{\xi_{ij}}(x_{\xi_{ij},\beta}^{(n_0)})
-F_{\alpha,n}^{\xi_{ij}}(x_{\xi_{ij},\alpha}^{(n)})\right|_p\leq\frac{1}{p^{n_0}}\max\limits_{i,j}
\left|x_{\xi_{ij},\beta}^{(n_0)}-x_{\xi_{ij},\alpha}^{(n)}\right|_p\leq\frac{1}{p^{n_0}}<\varepsilon.
\end{array}
\end{equation}
From \eqref{lim1} as $n\to\infty$ one gets
$$
|x_{\beta}^{(n_0)}-x_\alpha|_p<\varepsilon.
$$
It follows that $\pi$ is continuous. Hence,
$\widetilde{\Lambda}^\xi$ is compact. It is clear that
$\Lambda_0^\xi\subset\widetilde{\Lambda}^\xi\subset\Lambda^\xi. $
Therefore, due to closedness of $\widetilde{\Lambda}^\xi$ and
$\overline{\Lambda}_0^\xi=\Lambda^\xi$ we immediately find
$\Lambda^\xi=\widetilde{\Lambda}^\xi$. Consequently, $\Lambda$ is
compact. The proof is complete.
\end{proof}

\begin{lem}
Let
$F=\sum\limits_{i=1}^M\prod\limits_{j=1}^LF_{\alpha,n}^{\xi_{ij}}\in\cf_\xi$.
For any
$x=\sum\limits_{i=1}^M\prod\limits_{j=1}^Lx_{\xi_{ij},\beta}^{(\ell)}\in\Lambda_0^\xi$
let us define a mapping by
\begin{equation}\label{F[x]}
\widetilde{F}[x]:=\sum\limits_{i=1}^M\prod\limits_{j=1}^LF_{\alpha,n}^{\xi_{ij}}(x_{\xi_{ij},\beta}^{(\ell)}),
\end{equation}
Then one has $\widetilde{F}[\Lambda_0^\xi]\subset\Lambda^\xi$.
\end{lem}
\begin{proof}
Let us establish $\widetilde{F}[x]\in\Lambda^\xi$. Take any $r>0$
and $y\in B_r(F[x])$. Consider the following sequence
$$
F_m=\sum\limits_{i=1}^M\prod\limits_{j=1}^LF_{\alpha,n}^{\xi_{ij}}\circ\left(F_{\beta,l}^{\xi_{ij}}\right)^m.
$$
It is clear that $F_m\in\cf_\xi$ for all $m\geq1$. Then one gets
\begin{eqnarray*}
\left|F_m(y)-\widetilde{F}[x]\right|_p&=&\left|\sum\limits_{i=1}^M\prod\limits_{j=1}^LF_{\alpha,n}^{\xi_{ij}}\circ\left(F_{\beta,\ell}^{\xi_{ij}}\right)^m(y)
-\sum\limits_{i=1}^M\prod\limits_{j=1}^LF_{\alpha,n}^{\xi_{ij}}(x_{\xi_{ij},\beta}^{(\ell)})\right|_p\\
&\leq&\max\limits_{i,j}\left|F_{\alpha,n}^{\xi_{ij}}\circ\left(F_{\beta,\ell}^{\xi_{ij}}\right)^m(y)
-F_{\alpha,n}^{\xi_{ij}}(x_{\xi_{ij},\beta}^{(\ell)})\right|_p\\
&\leq&\frac{1}{p^n}\max\limits_{i,j}\left|(F_{\b,\ell}^{\xi_{ij}})^m(y)-x_{\xi_{ij},\b}^{(\ell)}\right|_p\\
&<&\max\limits_{i,j}\left|(F_{\b,\ell}^{\xi_{ij}})^m(y)-(F_{\b,\ell}^{\xi_{ij}})^m(x_{\xi_{ij},\b}^{(\ell)})\right|_p\leq\frac{1}{p^m}
\max\limits_{i,j}\left|y-x_{\xi_{ij},\b}^{(\ell)}\right|_p\leq\frac{1}{p^m}
\end{eqnarray*}
This means that there exists $m_r\in\bn$ such that
$$
F_m(B_r(F[x]))\cap B_r(x)\neq\emptyset\ \ \ \mbox{for all }m\geq m_r
$$
which yields that $\widetilde{F}[x]\in\Lambda^\xi$.
\end{proof}

\begin{rem}
Since $\Lambda_0^\xi\subset\Lambda^\xi$ a natural question arises: how can we extend the function \eqref{F[x]} to $\Lambda^\xi$?
\end{rem}
Given $\alpha,\beta\in\Sigma$ and $n\in\bn$ we define an element
of $\Sigma$ by
$$
\alpha^{[n]}\vee\beta=(\alpha_1,\dots,\alpha_n,\beta_1,\beta_2,\dots)
$$
\begin{lem}
For each $F_{\alpha,n}\in\cf_{\xi,n}$ and $\beta\in\Sigma$ the
sequence
$$
\left\{\sum_{i=1}^M\prod_{j=1}^LF_{\alpha,n}^{\xi_{ij}}(x_{\xi_{ij},\beta}^{(m)})\right\}_{m\in\bn}
$$
is Cauchy. Here, as before, $x_{\xi_{ij},\beta}^{(m)}$ is a fixed
point of $F_{\beta,m}^{\xi_{ij}}$.
\end{lem}
\begin{proof} The proof immediately follows from
\begin{eqnarray*}
\left|\sum\limits_i\prod\limits_jF_{\alpha,n}^{\xi_{ij}}(x_{\xi_{ij},\beta}^{(m)})-
\sum\limits_i\prod\limits_jF_{\alpha,n}^{\xi_{ij}}(x_{\xi_{ij},\beta}^{(l)})\right|_p&\leq&\max\limits_{i,j}
\left|F_{\alpha,n}^{\xi_{ij}}(x_{\xi_{ij},\beta}^{(m)})-F_{\alpha,n}^{\xi_{ij}}(x_{\xi_{ij},\beta}^{(l)})\right|_p\\[2mm]
&\leq&\frac{1}{p^n}\max\limits_{i,j}\left|x_{\xi_{ij},\beta}^{(m)}-x_{\xi_{ij},\beta}^{(l)}\right|_p\\[2mm]
&\leq&\frac{1}{p^{n+\min\{m,l\}}}\to0
\end{eqnarray*}
\end{proof}
For any $F=\sum\limits_{i=1}^M\prod\limits_{j=1}^LF_{\alpha,n}^{\xi_{ij}}\in\cf_{\xi}$ we define
\begin{equation}\label{F[x]2}
\widetilde{F}[x_{\beta}]:=\lim_{m\to\infty}\sum_{i=1}^M\prod_{j=1}^LF_{\alpha,n}^{\xi_{ij}}(x_{\xi_{ij},\beta}^{(m)})
\end{equation}
\begin{thm}\label{teor5}
For any
$F=\sum\limits_{i=1}^M\prod\limits_{j=1}^LF_{\alpha,n}^{\xi_{ij}}\in\cf_{\xi}$
one has
$$
\widetilde{F}[x_\beta]=x_{\alpha^{[n]}\vee\beta}.
$$
Moreover, $\widetilde{F}[\Lambda^\xi]\subset\Lambda^\xi$.
\end{thm}
\begin{proof}
By definition we have
\begin{equation}\label{x_alpha,beta}
x_{\alpha^{[n]}\vee\beta}=\lim_{m\to\infty}\sum_{i=1}^M\prod_{j=1}^Lx_{\xi_{ij},\alpha^{[n]}\vee\beta}^{(m)},
\end{equation}
where $x_{\xi_{ij},\alpha^{[n]}\vee\beta}^{(m)}$ is a fixed point of
$$
\underbrace{f_{\xi_{ij}(\alpha_1)}\circ\cdots\circ f_{\xi_{ij}(\alpha_n)}}_{F_{\alpha,n}^{\xi_{ij}}}
\circ\underbrace{f_{\xi_{ij}(\beta_1)}\circ\cdots\circ f_{\xi_{ij}(\beta_{m-n})}}_{F_{\beta,m-n}^{\xi_{ij}}}
$$
On the other hand, one has
\begin{equation}\label{1123}
\widetilde{F}[x_\beta]=\lim_{m\to\infty}\sum_{i=1}^M\prod_{j=1}^LF_{\alpha,n}^{\xi_{ij}}(x_{\xi_{ij},\beta}^{(m)})
\end{equation}
Hence, one gets
\begin{eqnarray*}
\left|\sum\limits_{i}\prod\limits_{j}F_{\alpha,n}^{\xi_{ij}}(x_{\xi_{ij},\beta}^{(m)})-
\sum\limits_{i}\prod\limits_{j}x_{\xi_{ij},\alpha^{[n]}\vee\beta}^{(m)}\right|_p
&\leq&\max\limits_{i,j}\left|F_{\alpha,n}^{\xi_{ij}}(x_{\xi_{ij},\beta}^{(m)})-x_{\xi_{ij},\alpha^{[n]}\vee\beta}^{(m)}\right|_p\\[2mm]
&=&\max\limits_{i,j}\left|F_{\alpha,n}^{\xi_{ij}}(x_{\xi_{ij},\beta}^{(m)})-F_{\alpha,n}^{\xi_{ij}}\circ F_{\beta,m-n}^{\xi_{ij}}
(x_{\xi_{ij},\alpha^{[n]}\vee\beta}^{(m)})\right|_p\\
&\leq&\frac{1}{p^n}\max\limits_{i,j}\left|F_{\beta,m}^{\xi_{ij}}(x_{\xi_{ij},\beta}^{(m)})-
F_{\beta,m-n}^{\xi_{ij}}(x_{\xi_{ij},\alpha^{[n]}\vee\beta}^{(m)})\right|_p\\
&\leq&\frac{1}{p^m}
\end{eqnarray*}
Hence, from \eqref{x_alpha,beta} and \eqref{1123} we find the
desired equality. This completes the proof.
\end{proof}

\begin{thm}
If $\Lambda^\xi$ contains at least two points than it is perfect.
\end{thm}

\begin{proof}
Let $\Lambda^\xi$ contain at least two points. Since
$\Lambda^\xi=\widetilde{\Lambda}^\xi$ and
$\overline{\Lambda}_0^\xi=\widetilde{\Lambda}^\xi$ it is enough to
show that each $x\in\Lambda_0^\xi$ is not isolated point of
$\widetilde{\Lambda}^\xi$.

Let $x\in\Lambda_0^\xi$. Then there exist $\a\in\Sigma$ and a
positive integer $n$ such that
$$
x=\sum_{i=1}^M\prod_{j=1}^Lx_{\xi_{ij},\a}^{(n)},
$$
where $x_{\xi_{ij},\a}^{(n)}$ is a fixed point of $F_{\a,n}^{\xi_{ij}}$. It is clear that
$$
x=\sum_{i=1}^M\prod_{j=1}^L\left(F_{\a,n}^{\xi_{ij}}\right)^m(x_{\xi_{ij},\a}^{(n)})\ \ \ \mbox{for all }m\geq1
$$
Take any $r>0$ and $y_\b\in\widetilde{\Lambda}^\xi\setminus
B_r(x)$. Choose a positive integer $m$ such that
$\frac{1}{p^m}<\frac{r}{2}$. Take $\g\in\Sigma$ such that
$\g_{jn+i}=\a_i$ for all $i=\overline{1,n}$ and
$j=\overline{0,m-1}$.

Note that
$$
F_{\g,mn}^{\xi_{ij}}=\underbrace{\left(f_{\xi_{ij}(\a_1)}\circ\dots\circ f_{\xi_{ij}(\a_n)}\right)\circ\dots\circ
\left(f_{\xi_{ij}(\a_1)}\circ\dots\circ f_{\xi_{ij}(\a_n)}\right)}_m=(F_{\a,n}^{\xi_{ij}})^m
$$
So,
$$
F=\sum_{i=1}^M\prod_{j=1}^LF_{\g,mn}^{\xi_{ij}}\in\cf_\xi.
$$
Due to Theorem \ref{teor5} we have
$$
\widetilde{F}[y_\b]=\lim_{k\to\infty}\sum_{i=1}^M\prod_{j=1}^LF_{\g,mn}^{\xi_{ij}}(y_{\xi_{ij},\b}^{(k)})
$$
and belongs to $\widetilde{\Lambda}^\xi$. Now choose $k\geq1$ such
that
\begin{equation}\label{perf1}
\left|\widetilde{F}[y_\b]-\sum_{i=1}^M\prod_{j=1}^LF_{\g,mn}^{\xi_{ij}}(y_{\xi_{ij},\b}^{(k)})\right|_p<\frac{r}{2}.
\end{equation}
One can see that
\begin{eqnarray}\label{perf2}
\widetilde{F}[y_\b]-x&=&\widetilde{F}[y_\b]-\sum_{i=1}^M\prod_{j=1}^LF_{\g,mn}^{\xi_{ij}}(y_{\xi_{ij},\b}^{(k)})
+\sum_{i=1}^M\prod_{j=1}^LF_{\g,mn}^{\xi_{ij}}(y_{\xi_{ij},\b}^{(k)})\nonumber\\
&&-\sum_{i=1}^M\prod_{j=1}^L(F_{\a,n}^{\xi_{ij}})^m(x_{\xi_{ij},\a}^{(n)})\nonumber\\[2mm]
&=&\widetilde{F}[y_\b]-\sum_i\prod_jF_{\g,mn}^{\xi_{ij}}(y_{\xi_{ij},\b}^{(k)})\nonumber\\
&&+\sum_{i}\sum_j\left[F_{\g,mn}^{\xi_{ij}}(y_{\xi_{ij},\b}^{(k)})
-F_{\g,mn}^{\xi_{ij}}(x_{\xi_{ij},\a}^{(n)})\right]
\prod_{l>j\atop{u<j}}F_{\g,mn}^{\xi_{il}}(y_{\xi_{il},\b}^{(k)})F_{\g,mn}^{\xi_{iu}}(x_{\xi_{iu},\a}^{(n)})
\end{eqnarray}
Noting
$$
\left|F_{\g,mn}^{\xi_{ij}}(y_{\xi_{ij},\b}^{(k)})
-F_{\g,mn}^{\xi_{ij}}(x_{\xi_{ij},\a}^{(n)})\right|_p\leq\frac{1}{p^{mn}}<\frac{1}{p^m}<\frac{r}{2}.
$$
and using \eqref{perf1} and the strong triangle inequality from
\eqref{perf2} we obtain
$$
\left|\widetilde{F}[y_\b]-x\right|_p<r
$$
Hence, $\widetilde{F}[y_\b]\in\widetilde{\Lambda}^\xi\cap B_r(x)$. This means that $x$ is not isolated point of
$\widetilde{\Lambda}^\xi$. Consequently, $\Lambda^\xi$ is perfect.
\end{proof}

\begin{thm}\label{dud} The limiting set $\Lambda^\xi$ is doubling and uniformly disconnected.
\end{thm}
\begin{proof}
Note that $\mathbb Q_p$ is a doubling. Since $\Lambda^\xi\subset\mathbb Z_p$, the limiting set also has doubling property.
For every $a\in\Lambda^\xi$ and $r>0$, let $A=B_r(a)$. Then $B_{r/2}(a)\subset A\subset B_r(a)$. For any
$y\in\Lambda^\xi\setminus B_r(a)$ and $x\in B_r(a)$, by the strong triangle equality, we have $|x-y|_p=|y -a|_p\geq r>r/2$,
namely $\mbox{dist}(B_r(a),\Lambda^\xi\setminus B_r(a))\geq r/2$, which shows that $\Lambda^\xi$ is uniformly disconnected.
\end{proof}

\section{examples}

In the section we provide certain examples of contractive mappings
for which the unconventional limiting set is quasi-symmetrically
equivalent to the symbolic Cantor set $F^\infty$. But, in general,
such kind of result is unknown.

 Let $f_1(x)=px$ and $f_2(x)=px+1-p$. It
is clear that $f_i$ is a contractive mapping on $\mathbb Z_p$. For
convenience, we denote
$$
f_{i_1i_2\dots i_n}:=f_{i_1}\circ f_{i_2}\circ\dots\circ f_{i_n}
$$
and by $x_{i_1i_2\cdots i_n}$ fixed point of the mapping
$f_{i_1i_2\dots i_n}$.
Note that $x_1=0$ and $x_2=1$.

For any $x\in\mathbb Z_p$ we get
$$
f_{i_1i_2\dots i_n}(x)=p^nx+(1-p)(\d_{2i_1}+\d_{2i_2}p+\cdots+\d_{2i_n}p^{n-1}),
$$
where $\d_{2i_k}$ is a Kroneker symbol. It follows that
\begin{equation}\label{x_iii}
x_{i_1i_2\cdots i_n}=\frac{\d_{2i_1}+\d_{2i_2}p+\cdots+\d_{2i_n}p^{n-1}}{1+p+p^2+\cdots+p^{n-1}}.
\end{equation}
Let $\Lambda_0=\{x_{i_1i_2\cdots i_n}: \ \{i_1,i2\dots
i_n\}\in\{1,2\}, \ n\in\bn\}$. It is clear that
$\Lambda_0\subset\mathbb Z_p$. Note that $x\in\Lambda_0$ if and
only if $1-x\in\Lambda_0$. Denote
$\Lambda:=\overline{\Lambda}_{0}$.
\begin{lem}  The set $\Lambda$ has the
following properties:
\begin{enumerate}
\item[$(\Lambda.1)$] $x\in\Lambda$ if and only if $x$ has the
following canonical form
$$
x=p^{\gamma(x)}(1+x_1\cdot p+x_2\cdot p^2+\dots),
$$
where $\gamma(x)\in\mathbb Z_+$ and $x_i\in\{0,p-1\},\ i=1,2,3,\dots$

\item[$(\Lambda.2)$] The set $\Lambda$ is uniformly perfect.
\end{enumerate}
\end{lem}

\begin{proof}
$(\Lambda.1)$ From \eqref{canonic} one can see that each
$x\in\mathbb Z_p$ has the following canonical form
$$
x=p^{\gamma(x)}(x_0+x_1\cdot p+x_2\cdot p^2+\cdots),
$$
where $\gamma(x)\in\mathbb Z_+,\ x_0\neq0$. Let $x\in\Lambda\setminus\{0\}$.

Since $\Lambda=\overline{\Lambda}_{0}$ one has  $x\in\Lambda$ if
and only if there exists a sequence $\{i_k\}_{k\in\mathbb N},\
i_k\in\{1,2\}$ such that
\begin{equation}\label{x-sum}
\left|p^{\gamma(x)}\bigg[\sum_{j=0}^n\sum_{i=0}^jx_ip^j+\sum_{j=1}^\infty\sum_{i=j}^{n+j}x_ip^{n+j}\bigg]-
\sum_{j=1}^{n+1}\delta_{2i_j}p^{j-1}\right|_p<\frac{1}{p^n}
\end{equation}
From \eqref{x-sum} one can see $i_k=1,\ k=\overline{1,\gamma(x)+1}$. So, without loss of generality we may assume that $\gamma(x)=0$.

Let $x\in\Lambda\setminus\{0\}$. Assume that $x_0\neq1$. Then for any sequence $\{i_k\}_{k\in\mathbb N}$ using non-Archimedean norm's property
we get
$$
\left|p^{\gamma(x)}\bigg[\sum_{j=0}^n\sum_{i=0}^jx_ip^j+\sum_{j=1}^\infty\sum_{i=j}^{n+j}x_ip^{n+j}\bigg]-
\sum_{j=1}^{n+1}\delta_{2i_j}p^{j-1}\right|_p=|x_0-\delta_{2i_1}|_p=1.
$$
This contradicts to $x\in\Lambda$. So, we conclude that $x_0=1$ if
$x\in\Lambda\setminus\{0\}$.

Now we will show that $x_i\in\{0,p-1\},\ i=1,2,3,\dots$ For a given sequence $\{i_k\}_{k\in\mathbb N}$ we define
the following sequence
$$
a_n(x)=\left|1+\sum_{j=1}^n\bigg(1+\sum_{i=1}^jx_i\bigg)p^j+\sum_{j=1}^\infty\sum_{i=j}^{n+j}x_ip^{n+j}-\sum_{j=1}^{n+1}\delta_{2i_j}p^{j-1}\right|_p.
$$
Note that if $x_1\not\in\{0,p-1\}$ then for any $\{i_k\}_{k\in\mathbb N}$ holds
$$
a_n(x)=\left|1-\delta_{2i_1}+(1+x_1-\delta_{2i_2})p\right|_p=\left\{\begin{array}{ll}
1, & \mbox{if }i_1=1\\
\frac{1}{p}, & \mbox{if }i_1=2.
\end{array}\right.
$$
This means that for any sequence $\{i_k\}$ one has
$a_n(x)\not\to0$ as $n\to\infty$. This contradicts to
$x\in\Lambda_1$.

Now let us assume that $x_i\in\{0,p-1\},\ i=\overline{1,k}$ and
$x_{k+1}\not\in\{0,p-1\}$. Redenote elements of $\{1,2,\dots,k\}$
by $x_{m_1}=x_{m_2}=\cdots=x_{m_s}=p-1$ and
$x_{m_{s+1}}=x_{m_{s+1}}=\cdots=x_{m_k}=0$. Then we have
\begin{eqnarray}\label{1+sum}
1+\sum_{j=1}^{k+1}\bigg(1+\sum_{i=1}^jx_i\bigg)p^j&=&1+\sum_{j=1}^{m_1-1}p^j+\sum_{j=m_1+1}^{m_2}p^{j}\nonumber\\
&+&(2p-1)\sum_{j=m_2}^{m_3-1}p^j+\cdots+(sp-s+1)\sum_{j=m_s}^{k-1}p^j\nonumber\\
&+&(sp-s+1)p^k+(sp-s+1+x_{k+1})p^{k+1}\nonumber\\
&=&1+\sum_{j\not\in\{m_1,m_2,\dots,m_s\}}^{k-1}p^j+p^k+(1+x_{k+1})p^{k+1}+sp^{k+2}.
\end{eqnarray}
From \eqref{1+sum} with the non-Archimedean norm's property, for
any sequence $\{i_j\}_{j\in\mathbb N}$ one gets
$$
a_n(x)\geq\frac{1}{p^{k+1}}.
$$
This contradicts to $x\in\Lambda_1$.

Let $x\in\mathbb Z_p$ has the following canonical form
$$
x=p^{\gamma(x)}(1+x_1\cdot p+x_2\cdot p^2+\cdots),
$$
where $\gamma(x)\in\mathbb Z_+$ and $x_i\in\{0,p-1\},\ i=1,2,3,\dots$. For any $n\geq1$ redenote elements
of $\{1,2,\dots,n\}$ such that $x_{m_1}=x_{m_2}=\cdots=x_{m_{s_n}}=p-1$
and $x_{m_{s_n+1}}=x_{m_{s_n+1}}=\cdots=x_{m_m}=0$. Pick a sequence $\{i_k\}_{k\in\mathbb N}$ such that
$i_k=1$ if $k\in\{1,2,\dots,\gamma(x)+1\}\cup\{m_1,m_2,\dots,m_{s_n}\}$ or $k\neq n$ and $i_k=2$ otherwise.
Then for this sequence we get $a_n(x)<p^{-n}$ which means that $x\in\Lambda_1$.

$(\Lambda.2)$ We will show that
for any $x\in\Lambda_1$ and $r>0$ there exists $y\in\Lambda_1$ such that
$$
\frac{r}{p}\leq|x-y|_p\leq r.
$$
According to $(\Lambda.1)$ it is enough to check only case $|x|_p>r$. Let us pick a positive integer $n$ such that
$p^{-1}r\leq p^{-n}\leq r$. Note that $\gamma(x)<n$.
Then take $y\in\Lambda_1$ such that $|y|_p=|x|_p$ and $y_i=x_i$ if $i<n-\gamma(x)$ and
$y_{n-\gamma(x)}\neq x_{n-\gamma(x)}$. Thus, we have $|x-y|_p=p^{-n}$.
\end{proof}

Let $ \xi=\left(\begin{array}{ll}
\xi_{11} & \xi_{12}\\
\xi_{21} & \xi_{22}
\end{array}\right)$
be a matrix which elements are $1$ or $2$. For $\xi_{ij}$ we
define an action on $\Sigma=\{1,2\}^{\mathbb N}$ as follows
$$
(\xi_{ij}\ast\alpha)_k=\left\{\begin{array}{ll}
1,&\mbox{if }\ \xi_{ij}+\alpha_k\ \mbox{is odd}\\
2,&\mbox{if }\ \xi_{ij}+\alpha_k\ \mbox{is even}.
\end{array}\right.
$$
For $f_1(x)=px$ and $f_2(x)=px+1-p$ and a given matrix $\xi$ by
$\Lambda^\xi$ we denote the unconventional limiting set of
$\mathcal F_\xi$. We will show that $\Lambda^\xi$ is
quasi-symmetrically equivalent to the symbolic Cantor set
$F^\infty$.

$1)$ Let $
\xi=\left(\begin{array}{ll}
2 & 2\\
2 & 2
\end{array}\right)$. In this case one has
$$
F_{\alpha,n}(x)=2\left(f_{a_1\a_2\dots\a_n}(x)\right)^2.
$$
It follows that
$$
\Lambda_0^\xi=\{2x_0^2: x_0\in\Lambda_0\}.
$$
Hence, $\Lambda^\xi=\{2x^2: x\in\Lambda\}.$

$2)$ Let $
\xi=\left(\begin{array}{ll}
2 & 1\\
2 & 1
\end{array}\right)$. In this case we get
$$
F_{\alpha,n}(x)=2f_{a_1a_2\dots\a_n}(x)\cdot f_{\overline{\a}_1\overline{\a}_2\dots\overline{\a}_n}(x),
$$
where
$$
\overline{\a}_i=\left\{\begin{array}{ll}
1,& \mbox{if }\a_i=2\\
2,& \mbox{if }\a_i=1.
\end{array}\right.
$$
Due to
$x_{\overline{\a}_1\overline{\a}_2\dots\overline{\a}_n}=1-x_{\a_1\a_2\dots\a_n}$,
one finds
$$\Lambda_0^\xi=\{2x_0(1-x_0): x_0\in\Lambda_0\}\ \ \ \mbox{and}\ \ \ \Lambda^\xi=\{2x(1-x): x\in\Lambda\}.$$

$3)$ Let $
\xi=\left(\begin{array}{ll}
2 & 2\\
1 & 1
\end{array}\right)$. In this case, we get
$$\Lambda_0^\xi=\{x_0^2+(1-x_0)^2: x_0\in\Lambda_0\}\ \ \ \mbox{and}\ \ \ \Lambda^\xi=\{x^2+(1-x)^2: x\in\Lambda\}.$$

$4)$ Let $
\xi=\left(\begin{array}{ll}
2 & 1\\
1 & 1
\end{array}\right)$. In this case, we get
$$\Lambda_0^\xi=\Lambda_0\ \ \ \mbox{and}\ \ \ \Lambda^\xi=\Lambda.$$

The uniformly perfectness of the sets $\Lambda^\xi$ immediately
follows from the uniformly perfectness of $\Lambda$. According to
 Theorems \ref{dud} and  \ref{qseC} we conclude that
$\Lambda^\xi$ is quasi-symmetrically equivalent to the symbolic
Cantor set $F^\infty$.


\end{document}